\documentclass[12pt]{article}
\usepackage{amssymb, amsmath}

\numberwithin{equation}{section}

\newtheorem{theorem}{Theorem}[section]
\newtheorem{proposition}[theorem]{Proposition}

\newtheorem{corollary}[theorem]{Corollary}

\newtheorem{lemma}[theorem]{Lemma}

\newcommand{\mc}{\mathcal}
\DeclareMathOperator{\lk}{lk}

\newcommand{\p}{^{\prime}}

\newcommand{\R}{\mathbb{R}}
\newcommand{\cc}{^{(c)}}
\newcommand{\scc}{^{(sc)}}

\title{Lower Bounds for Cubical Pseudomanifolds}
\author{Steven Klee \\
\small Department of Mathematics\\[-0.8ex]
\small One Shields Avenue \\[-0.8ex]
\small University of California \\[-0.8ex]
\small  Davis, CA, 95616, USA\\[-0.8ex]
\small \texttt{klee@math.ucdavis.edu}}

\begin{document}

\maketitle

\begin{abstract}
It is verified that the number of vertices in a $d$-dimensional cubical pseudomanifold is at least $2^{d+1}$.  Using Adin's cubical $h$-vector,  the generalized lower bound conjecture is established for all cubical $4$-spheres, as well as for some special classes of cubical spheres in higher dimensions.
\end{abstract}


\section{Introduction} \label{intro} 

In this paper, we will be interested in studying combinatorial properties of cubical decompositions of certain topological spaces.  We will begin with a brief discussion of simplicial complexes, which will hopefully help motivate the results presented in the remainder of the paper.  One of the most commonly studied combinatorial invariants of a finite $(d-1)$-dimensional simplicial complex $\Delta$ is its $f$-vector $f(\Delta) = (f_{-1}(\Delta), f_0(\Delta), \ldots, f_{d-1}(\Delta))$ where $f_i(\Delta)$ denotes the number of $i$-dimensional faces of $\Delta$.  Equivalently, the $h$-numbers of $\Delta$ are defined by the relation $\sum_{j=0}^dh_j(\Delta)\lambda^{j} = \sum_{i=0}^df_{i-1}(\Delta) \lambda^i (1-\lambda)^{d-i}$.

For a $(d-1)$-dimensional semi-Eulerian simplicial complex $\Delta$ (for example, a triangulation of a manifold without boundary), Vic Klee's \cite{VK} Dehn-Sommerville equations establish a symmetry among the $h$-numbers of $\Delta$.  Specifically, $h_{d-i}(\Delta)-h_i(\Delta) = (-1)^i{d \choose i}(\widetilde{\chi}(\Delta)-\widetilde{\chi}(\mathbb{S}^{d-1}))$ for all $0 \leq i \leq d$.  In particular, when $\Delta$ is a homology sphere, the $h$-vector of $\Delta$ is symmetric. More generally, Novik and Swartz \cite{NS08} establish an analogue of Klee's Dehn-Sommerville equations for manifolds with boundary (see Theorem \ref{NS-DehnSommerville}).

The $g$-numbers of a $(d-1)$-dimensional simplicial complex $\Delta$ are defined as $g_i(\Delta):=h_i(\Delta)-h_{i-1}(\Delta)$ for all $0 \leq i \leq \lfloor\frac{d}{2}\rfloor$.  McMullen's $g$-conjecture \cite{Mc71}, now the celebrated $g$-theorem of Stanley \cite{S80} and Billera and Lee \cite{BL}, provides a complete characterization of the $h$-numbers of a simplicial polytope (see Theorem \ref{g-theorem}).  One consequence of the $g$-theorem is the following \textit{generalized lower bound theorem}: if $\mathcal{P}$ is a simplicial $d$-polytope, then $h_0(\mc{P}) \leq h_1(\mc{P}) \leq \cdots \leq h_{\lfloor \frac{d}{2}\rfloor}(\mc{P})$. 

It is natural to ask if the $g$-theorem continues to hold for the class of $(d-1)$-dimensional simplicial spheres.  Unfortunately, it is not even known if the generalized lower bound theorem continues to hold for this class of simplicial complexes.  It is known, however, that if $\Delta$ is a $(d-1)$-dimensional simplicial sphere, then $h_0(\Delta) \leq h_1(\Delta) \leq h_2(\Delta)$.  The first of these inequalities is trivial, while the second is equivalent the following simplicial lower bound theorem. 

Let $\varphi_i(n,d)$ denote the number of $i$-dimensional faces in a stacked simplicial $d$-polytope on $n$-vertices. The lower bound conjecture (LBC, for short) states that if $\Delta$ is a $d$-polytope on $n$ vertices, then $f_i(\Delta) \geq \varphi_i(n,d)$ for all $i$.  Barnette \cite{B73} proved that $h_2(\Delta) \geq h_1(\Delta)$ for any simplicial $d$-polytope $\Delta$ with $d \geq 3$.  The inequality $h_2(\Delta) \geq h_1(\Delta)$ is equivalent to the inequality $f_1(\Delta) \geq \varphi_1(n,d)$ when $\Delta$ has $n$ vertices; but, as Perles observed and Barnette \cite{B73} proved, the lower bound conjecture actually follows from this inequality.  More generally, Kalai \cite{K87} proved that when $d \geq 4$, any $(d-1)$-dimensional homology manifold without boundary is generically $d$-rigid and hence also satisfies $h_2(\Delta) \geq h_1(\Delta)$. In particular, the LBC holds for homology spheres and homology manifolds without boundary.

In addition to studying simplicial complexes, one may also study cubical complexes.  Where simplicial complexes can be decomposed into simplices of certain dimensions, cubical complexes are geometric objects that can be decomposed into cubes of certain dimensions (see Section \ref{notation} for a formal definition).  Let $\mc{K}$ be a $d$-dimensional cubical complex.  We compute the $f$-vector $(f_0(\mc{K}),\ldots,f_d(\mc{K}))$ of $\mc{K}$, where $f_i(\mc{K})$ denotes the number of $i$-dimensional faces of $\mc{K}$.  Adin \cite{A95} defines a cubical $h$-vector $h^{(c)}(\mc{K}) = (h_0^{(c)}(\mc{K}),h_1^{(c)}(\mc{K}),\ldots,h_{d+1}^{(c)}(\mc{K}))$ that is analogous to the $h$-vector of a simplicial complex.  He goes on to prove an analogue of the Dehn-Sommerville equations for semi-Eulerian cubical complexes (see Theorem \ref{adin-ds}).

As in the case of simplicial complexes, the $g$-numbers of a cubical complex  $\mc{K}$ are defined by the formula $g_i^{(c)}(\mc{K}) = h_i^{(c)}(\mc{K})-h_{i-1}^{(c)}(\mc{K})$.  Adin \cite[Question 2]{A95} poses the following generalized lower bound conjecture (GLBC) for cubical polytopes:  if $\mc{K}$ is a cubical $(d+1)$-polytope, is $g_i^{(c)}(\mc{K}) \geq 0$ for all $1 \leq i \leq \lfloor\frac{d+1}{2}\rfloor$?  

Blind and Blind \cite{BB90} study shellings of cubical polytopes to show that a cubical polytope always has a pair of disjoint facets.  This implies that a cubical $(d+1)$-polytope has at least $2^{d+1}$ vertices.  This result answers a conjecture of Kupitz \cite{K84} asserting that the  $(d+1)$-cube has the minimal $f$-vector among all cubical $(d+1)$-polytopes.  Jockusch \cite{J93} poses a lower bound conjecture for cubical polytopes.  As in the simplicial case, his conjecture is equivalent to showing that $g_2\cc(\mc{K}) \geq 0$ for cubical polytopes.

In this paper, we study $d$-dimensional cubical complexes $\mc{K}$ with the property that each codimension-one face of $\mc{K}$ is contained in exactly two facets.  In addition to the class of cubical polytopes, this includes the more general classes of cubical spheres and cubical manifolds without boundary.  The first goal of this paper is to prove that if $\mc{K}$ is such a cubical complex, then $f_0(\mc{K}) \geq 2^{d+1}$ (Theorem \ref{g1bound}).  This provides a new proof of the result of Blind and Blind \cite{BB90} for the more general class of cubical pseudomanifolds, which makes no reference to shellings.

From here, we study the GLBC for cubical polytopes and spheres.  We show that $g_2\cc(\mc{K}) \geq 0$ for any $4$-dimenisonal cubical sphere $\mc{K}$ (Theorem \ref{4dim}), and hence cubical $4$-spheres satisfy the GLBC.  Next, we show that the cubical GLBC is satisfied for a class of odd-dimensional cubical polytopes (Theorem \ref{g-thm-small-links}).  The second goal of the paper is to prove an analogue of the Dehn-Sommerville equations for cubical manifolds with boundary (Theorem \ref{cubicalDS}) that is analogous to a result of Novik and Swartz \cite{NS08} for simplicial manifolds with boundary.

The paper is structured as follows.  In Section \ref{notation}, we present the definitions and background material that will be relevant for the remainder of the paper.  In Section \ref{lowerbounds}, we study lower bounds for cubical complexes, proving Theorem \ref{g1bound}, along with some special cases of the cubical GLBC.  Section \ref{cubicalDSsection} is devoted to the proof of the cubical Dehn-Sommerville equations for manifolds with boundary.


\section{Definitions and Preliminary Results} \label{notation}
Let $C^n$ denote the standard cube $[0,1]^n$ in $\R^n$.  A \textit{cubical complex} $\mc{K}$ on vertex set $V$ is a collection of subsets of $V$, partially ordered by inclusion, satisfying the following properties:
\begin{enumerate}
\item $\mc{K}$ has a minimal element $\emptyset$.
\item For all $v \in V$, the singleton $\{v\} \in \mc{K}$.
\item For any nonempty $F \in \mc{K}$, the interval $[\emptyset,F] = \{G \in \mc{K}: \emptyset \subseteq G \subseteq F\}$ is isomorphic to the face poset of a cube of some dimension.
\item If $F,G \in \mc{K}$, then $F \cap G \in \mc{K}$.
\end{enumerate}

The elements $F \in \mc{K}$ are called \textit{faces}.  By convention, we use the notation $G \subset F$ to indicate that $G$ is a proper (possibly empty) subface of $F$ and $G \subseteq F$ to indicate that $G$ is a subface of $F$.  If $[\emptyset,F]$ is isomorphic to the face poset of a $k$-dimensional cube, we say that $F$ is a \textit{$k$-dimensional} face of $\mathcal{K}$ and write $\dim F = k$.  This makes $\mc{K}$ a graded poset when we declare that a $k$-dimensional face of $\mc{K}$ has rank $k+1$.  We define the dimension of $\mc{K}$ to be $\dim(\mc{K}) := \max\{\dim F: F \in \mc{K}\}$.  We say that $\mc{K}$ is \textit{pure} of dimension $d$ if each of its facets (maximal faces under inclusion) has dimension $d$.  If $\mc{K}$ is pure, we call a codimension-one face of $\mc{K}$ a \textit{ridge}.

The $\textit{link}$ of a face $F \in \mc{K}$ is $\lk_{\mc{K}}(F) = \{G \in \mc{K}: G \supseteq F\}.$  If $\mc{K}$ is a pure $d$-dimensional cubical complex, and $F \in \mc{K}$ is a nonempty face of dimension $k$, then $\lk_{\mc{K}}(F)$ is the face poset of a \textit{simplicial} complex of dimension $d-k-1$.

As in the case of simplicial complexes, we define the $f$-vector of a $d$-dimensional cubical complex $\mc{K}$ to be $f(\mc{K}) = \left(f_{-1}(\mc{K}),f_0(\mc{K}),f_1(\mc{K}), \ldots, f_d(\mc{K})\right)$ where $f_i(\mc{K})$ denotes the number of $i$-dimensional faces of $\mc{K}$.  Adin \cite{A95}, defines the \textit{short cubical $h$-vector} of $\mc{K}$ to be the vector $h\scc(\mc{K}) = (h_0\scc(\mc{K}),\ldots,h_d\scc(\mc{K}))$ whose entries are defind by the relation $$\sum_{j=0}^dh_j\scc(\mc{K})\lambda^j = \sum_{i=0}^df_{i}(\mc{K})(2\lambda)^i(1-\lambda)^{d-i}.$$  Equivalently, Hetyei \cite[Theorem 9]{A95} observed that $h_j\scc(\mc{K}) = \sum_{v\in\mc{K}}h_j(\lk_{\mc{K}}(v))$, where the sum is taken over all vertices $v \in \mc{K}$, and $h_j(\lk_{\mc{K}}(v))$ denotes the simplicial $h$-number of  $\lk_{\mc{K}}(v)$.  Having defined the short cubical $h$-vector, Adin defines the \textit{(long) cubical $h$-vector} of $\mc{K}$ to be the vector $h\cc(\mc{K}) = (h_0\cc(\mc{K}),\ldots,h_{d+1}\cc(\mc{K}))$ where $h_0\cc(\mc{K}) = 2^d$ and $h_{i+1}\cc(\mc{K})+h_i\cc(\mc{K}) = h_i\scc(\mc{K})$ for all $i\geq 0$.  For example, $h_1^{(c)}(\mc{K}) = f_0(\mc{K})-2^d$.  Since the link of a vertex in a cubical complex $\mc{K}$ is a  simplicial complex, our primary means of understanding $h\cc(\mc{K})$ is to study $h\scc(\mc{K})$ and employ the plethora of known results about $h$-numbers of simplicial complexes.

Following Goresky and MacPherson \cite{GM}, we say that a pure $(d-1)$-dimensional simplicial complex $\Delta$ is a \textit{pseudomanifold} if each ridge of $\Delta$ is contained in exactly two facets.  Similarly, a pure $d$-dimensional cubical complex $\mc{K}$ is a \textit{pseudomanifold} if each ridge of $\mc{K}$ is contained in exactly two facets.

The \textit{reduced Euler characteristic} of a simplicial (or cubical) complex $\Gamma$ is $\widetilde{\chi}(\Gamma) := \sum_{i=-1}^{\dim \Gamma}(-1)^if_i(\Gamma)$.  A pure $(d-1)$-dimensional simplicial complex $\Delta$ is called \textit{semi-Eulerian} if $\widetilde{\chi}(\lk_{\Delta}(F)) = \widetilde{\chi}(\mathbb{S}^{d-|F|-1})$ for all nonempty faces $F \in \Delta$.  Klee's Dehn-Sommerville equations \cite{VK} state that if $\Delta$ is a semi-Eulerian simplicial complex, then $h_{d-i}(\Delta)-h_i(\Delta) = (-1)^i{d \choose i}(\widetilde{\chi}(\Delta)-\widetilde{\chi}(\mathbb{S}^{d-1}))$ for all $0 \leq i \leq d$. Similarly, we say that a $d$-dimensional cubical complex $\mc{K}$ is \textit{semi-Eulerian} if $\widetilde{\chi}(\lk_{\mc{K}}(F)) = \widetilde{\chi}(\mathbb{S}^{d-\dim F-1})$ for all nonempty faces $F \in \mc{K}$, and that $\mc{K}$ is Eulerian if it is semi-Eulerian and $\widetilde{\chi}(\mc{K}) = \widetilde{\chi}(\mathbb{S}^d)$.

Adin \cite[Theorem 5(i)]{A95} proves that if $\mc{K}$ is a semi-Eulerian cubical complex of dimension $d$, then $h_i\scc(\mc{K}) = h_{d-i}\scc(\mc{K})$ for all $0 \leq i \leq d$.  The following \textit{cubical Dehn-Sommerville equations} follow from the relation $h_{j+1}\cc(\mc{K})+ h_{j}\cc(\mc{K}) = h_j\scc(\mc{K})$. 

\begin{theorem}\label{adin-ds}{\rm (essentially Adin \cite[Theorem 5(ii)]{A95})} 
Let  $\mc{K}$ be a semi-Eulerian cubical complex of dimension $d$.  Then 
\begin{equation}\label{dehn-sommerville}
h_{d+1-i}\cc(\mc{K})-h_i\cc(\mc{K}) = (-1)^i(-2)^d(\widetilde{\chi}(\mc{K}) - \widetilde{\chi}(\mathbb{S}^d)), 
\end{equation}
for all $0 \leq i \leq d+1$.
\end{theorem}

We conclude this section with a statement of the $g$-theorem for simplicial polytopes, as it will be used in Section \ref{lowerbounds}.  For further information, we refer the reader to \cite{S96}.  Henceforth, if $\mathcal{P}$ is a simplicial (or cubical) $d$-polytope, we will use the notation $h_j(\mathcal{P})$ (or $h\scc_j(\mathcal{P})$, $ h\cc_j(\mathcal{P})$) to refer to the $h$-numbers (or cubical $h$-numbers) of the $(d-1)$-dimensional simplicial (or cubical) boundary complex of $\mathcal{P}$.  

Given nonnegative integers $\ell$ and $i$, it is possible to find nonnegative integers $n_i>n_{i-1}>\ldots>n_s\geq s \geq 1$ for which $\ell = {n_i \choose i} + {n_{i-1} \choose i-1} + \cdots + {n_s \choose s}$.  Given such an expansion, define $$\ell^{\langle i \rangle}:= {n_i+1 \choose i+1} + {n_{i-1} \choose i} + \cdots + {n_s+1 \choose s+1}.$$ Recall that if $\Delta$ is a $(d-1)$-dimensional simplicial complex, the \textit{$g$-numbers} of $\Delta$ are defined by $g_0(\Delta)=1$, and $g_i(\Delta) = h_i(\Delta)-h_{i-1}(\Delta)$ for all $1 \leq i \leq \lfloor \frac{d}{2}\rfloor$.  

\begin{theorem}\label{g-theorem}{\rm (Stanley \cite{S80}, Billera-Lee \cite{BL})}
A vector $h = (h_0, \ldots, h_d) \in \mathbb{Z}^{d+1}$ is the $h$-vector of a simplicial $d$-polytope if and only if
\begin{enumerate}
\item \label{g-theorem-1} $h_0=1$; 
\item \label{g-theorem-2} $h_i = h_{d-i}$ for all $i$; and 
\item \label{g-theorem-3} $(h_0,h_1-h_0,\cdots,h_{\lfloor \frac{d}{2}\rfloor}-h_{\lfloor \frac{d}{2}\rfloor-1}):=(1,g_1,\cdots,g_{\lfloor\frac{d}{2}\rfloor})$ satisfies $0 \leq g_{i+1}\leq g_i^{\langle i\rangle}$ for all $i \geq 1$.  
\end{enumerate}
\end{theorem}
Any vector satisfying the conditions of Theorem \ref{g-theorem}.\ref{g-theorem-3} is called an \textit{M-vector} (see \cite[Theorem II.2.2]{S96}).


\section{Lower Bounds} \label{lowerbounds}

Our first goal is to establish that the number of vertices in a $d$-dimensional cubical pseudomanifold is at least $2^{d+1}$.  We begin with a lemma.

\begin{lemma} \label{bound_f_nums}
Let $\mathcal{K}$ be a $d$-dimensional cubical complex with $d \geq 1$.  Then
\begin{displaymath}
  \sum_{i=0}^d2^if_i(\mc{K}) \leq \left(f_0(\mc{K})\right)^2.
\end{displaymath}
\end{lemma}
\begin{proof}
Let $F \in \mc{K}$ be an $i$-dimensional face with $1 \leq i \leq d$.  There are $2^{i-1}$ distinct pairs of antipodal vertices $v,v\p \in F$ whose least upper bound in $\mc{K}$ is $F$.  On the other hand, any two distinct vertices $u,u\p \in \mc{K}$ that have an upper bound in $\mc{K}$ have a unique least upper bound in $\mc{K}$.  Thus
\begin{displaymath}
\sum_{i=1}^d2^{i-1}f_i(\mc{K}) \leq {f_0(\mc{K}) \choose 2},
\end{displaymath}
which gives the desired inequality.
\end{proof}

This seemingly innocent lemma has the following consequence, which generalizes a result of Blind and Blind \cite[Theorem 2]{BB90}.

\begin{theorem} \label{g1bound}
Let $\mc{K}$ be a $d$-dimensional cubical pseudomanifold with $d \geq 2$.  Then $f_0(\mc{K}) \geq 2^{d+1}$.
\end{theorem}
\begin{proof}
Observe that $2^if_i(\mc{K}) = \sum_{v \in \mc{K}}f_{i-1}(\lk_{\mc{K}}(v))$ for any $1 \leq i \leq d$.  Indeed, both quantities count the number of pairs $(v,F)$ such that $F \in \mc{K}$ is an $i$-dimensional face containing the vertex $v$.  Since $\mc{K}$ is a pseudomanifold, each vertex of $\mc{K}$ is contained in at least $d+1$ facets of $\mc{K}$.  Hence $f_{i-1}(\lk_{\mc{K}}(v)) \geq {d+1 \choose i}$ for all $0 \leq i \leq d$ by the Kruskal-Katona theorem \cite[Theorem II.2.1]{S96}; and moreover, equality holds if and only if the vertex $v$ is contained in exactly $d+1$ facets.  Thus
\begin{equation}\label{eq2}
\sum_{i=0}^d 2^if_i(\mc{K}) = \sum_{v \in \mc{K}} \sum_{i=0}^d f_{i-1}(\lk(v)) \geq \sum_{v \in \mc{K}} \sum_{i=0}^d {d+1 \choose i} = f_0(\mc{K})\cdot(2^{d+1}-1),
\end{equation}
and hence $f_0(\mc{K}) \geq 2^{d+1}-1$ by Lemma \ref{bound_f_nums}.  

Suppose by way of contradiction that $f_0(\mc{K}) = 2^{d+1}-1$.  Then equality holds at each step of Equation \eqref{eq2}, and each vertex $v \in \mc{K}$ is contained in exactly $d+1$ facets of $\mc{K}$. This gives $$2^df_d(\mc{K}) = \sum_{v \in \mc{K}} f_{d-1}(\lk_{\mc{K}}(v)) = (d+1)f_0(\mc{K}) = (d+1)(2^{d+1}-1).$$  When $d \geq 2$,  $2^d$ divides the left hand side of the above equation but not the right side, a contradiction.  Thus $f_0(\mc{K}) \geq 2^{d+1}$, as desired.
\end{proof}

In other words, Theorem \ref{g1bound} establishes that $h_1\cc(\mc{K}) \geq h_0\cc(\mc{K})$ for cubical pseudomanifolds.  As a corollary, we remark that $f$-vectors of $d$-dimensional cubical pseudomanifolds are minimized by the boundary complex of the $(d+1)$-cube.

\begin{corollary}
Let $\mc{K}$ be a $d$-dimensional cubical pseudomanifold with $d \geq 2$. Then for all $0\leq i\leq d$,
\begin{displaymath}
f_i(\mc{K}) \geq {d+1 \choose i}2^{d+1-i}.
\end{displaymath}
\end{corollary}

\begin{proof}
Since $f_{i-1}(\lk_{\mc{K}}(v)) \geq {d+1 \choose i}$, it follows that
\begin{displaymath}
2^if_i(\mc{K}) = \sum_{v \in K} f_{i-1}(\lk_{\mc{K}}(v)) \geq 2^{d+1}{d+1 \choose i}.
\end{displaymath}
\end{proof}

Jockusch \cite{J93} poses a lower bound conjecture for cubical polytopes.  As in the case of simplicial polytopes, the cubical lower bound conjecture states that $h_1\cc(\mc{K}) \leq h_2\cc(\mc{K})$ when $\mc{K}$ is a cubical $(d+1)$-polytope and $d \geq 3$.  When $n$, the number of vertices in $\mc{K}$, is divisible by $2^d$, this says that a stacked cubical $(d+1)$-polytope on $n$ vertices has the minimal $f$-numbers among all cubical  $(d+1)$-polytopes on $n$ vertices.  A stacked cubical polytope on $n$ vertices is constructed by stacking an appropriate number of combinatorial $(d+1)$-cubes on top of one another.  In the remainder of this section, we verify that the cubical GLBC holds in some special cases.  We begin by proving this conjecture for all $4$-dimensional cubical spheres (see Theorem \ref{4dim}).

\begin{lemma}\label{cubical_shortcubical_g-s}
Let $\mc{K}$ be a $d$-dimensional cubical complex.  For all $0 \leq i \leq d$,
$$h_{i+1}\cc(\mc{K}) = (-1)^{i+1}h_0\cc(\mc{K})+\sum_{j=0}^{i}(-1)^{i-j}h_{j}\scc(\mc{K}).$$  In particular, 
\begin{equation}\label{lem35}
h_{i+1}\cc(\mc{K})-h_{i-1}\cc(\mc{K}) = h_i\scc(\mc{K})-h_{i-1}\scc(\mc{K}),
\end{equation}
for all $1 \leq i \leq d$.
\end{lemma}
\begin{proof}
We use the relation $h_i\scc(\mc{K}) = h_i\cc(\mc{K}) + h_{i+1}\cc(\mc{K})$ and induction on $i$.  The result is immediate when $i=0$.  For $i\geq1$, using the inductive hypothesis we obtain
\begin{eqnarray*}
h_{i+1}\cc(\mc{K}) &=&h_i\scc(\mc{K}) - h_i\cc(\mc{K})  \\
&=& h_i\scc(\mc{K}) + (-1)^{i+1}h_0\cc(\mc{K}) + \sum_{j=0}^{i-1}(-1)^{i-j}h_j\scc(\mc{K})  \\
&=& (-1)^{i+1}h_0\cc(\mc{K}) + \sum_{j=0}^i(-1)^{i-j}h_j\scc(\mc{K}).
\end{eqnarray*}
\end{proof}

\begin{corollary} \label{even-stacked}
Let $\mc{K}$ be a $d=2k$-dimensional Eulerian cubical complex with the property that $\lk_{\mc{K}}(v)$ is a stacked simplicial sphere for all vertices $v \in \mc{K}$.  Then $h_1\cc(\mc{K}) = h_2\cc(\mc{K}) = \cdots = h_d\cc(\mc{K}).$
\end{corollary}
\begin{proof}
Since $\lk_{\mc{K}}(v)$ is stacked for all vertices $v \in \mc{K}$, it follows that $h_1(\lk_{\mc{K}}(v)) = h_2(\lk_{\mc{K}}(v)) = \cdots = h_{d-1}(\lk_{\mc{K}}(v))$ for all $v \in \mc{K}$, and hence that $h_1\scc(\mc{K}) = \cdots = h_{d-1}\scc(\mc{K})$.  Thus Equation \eqref{lem35}, implies that $h_{i+1}\cc(\mc{K}) = h_{i-1}\cc(\mc{K})$ for all $2 \leq i \leq d-1$.  Since $\mc{K}$ is Eulerian, $h_{k+1}\cc(\mc{K}) = h_k\cc(\mc{K})$ by the cubical Dehn-Sommerville equations \eqref{dehn-sommerville}, and the statement follows.
\end{proof}

\begin{theorem} \label{4dim}
Let $\mc{K}$ be a $4$-dimensional cubical sphere.  Then $g_2\cc(\mc{K})~\geq~0$.
\end{theorem}
\begin{proof}
By the cubical Dehn-Sommerville equations \eqref{dehn-sommerville}, $h_3\cc(\mc{K}) = h_2\cc(\mc{K})$; and by Equation \eqref{lem35}, $h_3\cc(\mc{K})-h_1\cc(\mc{K}) = h_2\scc(\mc{K})-h_1\scc(\mc{K})$.  Thus $h_2\cc(\mc{K})-h_1\cc(\mc{K}) = h_2\scc(\mc{K})-h_1\scc(\mc{K})$.  

Since the link of each vertex of $\mc{K}$ is a $3$-dimensional simplicial sphere, the simplicial LBT implies that $h_2\scc(\mc{K})-h_1\scc(\mc{K}) = \sum_vh_2(\lk_{\mc{K}}(v))-h_1(\lk_{\mc{K}}(v)) \geq 0$.  Thus $h_2\cc(\mc{K})-h_1\cc(\mc{K})~\geq~0$.
\end{proof}

The proof of Theorem \ref{4dim} can be modified to prove the following result for odd-dimensional cubical polytopes (whose boundary complexes are even-dimensional cubical spheres) by appealing to the $g$-theorem. 

\begin{proposition}
Let $\mc{K}$ be a cubical polytope of dimension $d=2k+1$.  Then $g_k\cc(\mc{K})~\geq~0$.  
\end{proposition}
\begin{proof}
By the cubical Dehn-Sommerville equations \eqref{dehn-sommerville}, $h_{k+1}\cc(\mc{K}) = h_k\cc(\mc{K})$.  By Equation \eqref{lem35}, $h_{k+1}\cc(\mc{K})-h_{k-1}\cc(\mc{K}) = h_k\scc(\mc{K})-h_{k-1}\scc(\mc{K})$.  The link of each vertex in $\mc{K}$ is polytopal, and hence satisfies the $g$-theorem.  Thus $h_k\scc(\mc{K}) \geq h_{k-1}\scc(\mc{K})$, and the result follows.
\end{proof}

Now we will turn our attention to some partial results on even-dimensional cubical spheres.

\begin{lemma}\label{cubical_g_even_dimension}
Let $\mc{K}$ be a $d=2k$-dimensional cubical sphere.  Then
\begin{displaymath}
h_i\cc(\mc{K})-h_{i-1}\cc(\mc{K}) = \sum_{j=i}^k(-1)^{j-i}g_j\scc(\mc{K}), \qquad \text{ for all } 1 \leq i \leq k.
\end{displaymath}
\end{lemma}
\begin{proof}
We use Equation \eqref{lem35} and the cubical Dehn-Sommerville equations \eqref{dehn-sommerville} to obtain the following.
\begin{eqnarray*}
\sum_{j=i}^k(-1)^{j-i}g_j\scc(\mc{K}) &=& \sum_{j=i}^k(-1)^{j-i}(h_{j+1}\cc(\mc{K})-h_{j-1}\cc(\mc{K})) \\
&=& (-1)^{k-i}(h_{k+1}\cc(\mc{K})-h_k\cc(\mc{K})) + h_i\cc(\mc{K})-h_{i-1}\cc(\mc{K})\\
&=&h_i\cc(\mc{K})-h_{i-1}\cc(\mc{K}).
\end{eqnarray*}
\end{proof}

\begin{theorem}\label{g-thm-small-links}
Let $\mc{K}$ be a $(2k+1)$-dimensional polytope, and suppose that $g_2(\lk_{\mc{K}}(v)) \leq 2$ for all $v \in \mc{K}$.  Then $\mc{K}$ satisfies the cubical GLBC: $h_0\cc(\mc{K}) \leq h_1\cc(\mc{K}) \leq h_2\cc(\mc{K}) \leq \ldots \leq h_k\cc(\mc{K})$.
\end{theorem}
\begin{proof}
For all vertices $v \in \mc{K}$, the link of $v$ in $\mc{K}$ is a simplicial polytope, and hence the $g$-theorem holds for $\lk_{\mc{K}}(v)$.  In particular the vector $(1,g_1(\lk_{\mc{K}}(v)),\ldots,g_k(\lk_{\mc{K}}(v)))$ is an $M$-vector.  As such, expressing each $g_i(\lk_{\mc{K}}(v))$ in the form $$g_i(\lk_{\mc{K}}(v)) = {n_i \choose i} + {n_{i-1} \choose i-1} + \ldots +{n_s \choose s}$$ with $n_i>n_{i-1}>\ldots>n_s\geq s \geq1$, we obtain $$0 \leq g_{i+1}(\lk_{\mc{K}}(v)) \leq \left(g_i(\lk_{\mc{K}}(v))\right)^{\langle i\rangle} = {n_i+1 \choose i+1} + {n_{i-1}+1 \choose i} + \ldots + {n_s+1 \choose s+1}.$$  Since $g_2(\lk_{\mc{K}}(v)) \leq 2 = {2 \choose 2}+{1 \choose 1}$, it follows that $g_3(\lk_{\mc{K}}(v)) \leq {3 \choose 3} + {2 \choose 2}=2$.  In fact, by repeating this argument, we see that $g_i(\lk_{\mc{K}}(v)) \leq 2$ for all $i$.  Moreover, if $g_i(\lk_{\mc{K}}(v)) = 1 = {i \choose i}$ for some $i$, then $g_{i+1}(\lk_{\mc{K}}(v)) \leq {i+1 \choose i+1} = 1$ as well; and if $g_i(\lk_{\mc{K}}(v))=0$, then $g_{i+1}(\lk_{\mc{K}}(v)) = 0$.  Therefore, $$0 \leq g_{k}(\lk_{\mc{K}}(v)) \leq \cdots \leq g_3(\lk_{\mc{K}}(v)) \leq g_2(\lk_{\mc{K}}(v)) \leq 2,$$ for all $v \in \mc{K}$.  Thus, by Lemma \ref{cubical_g_even_dimension},
\begin{eqnarray*}
h_i\cc(\mc{K})-h_{i-1}\cc(\mc{K}) &=& \sum_{j=i}^k(-1)^{j-i}g_j\scc(\mc{K}) = \sum_v\sum_{j=i}^k(-1)^{j-i}g_j(\lk_{\mc{K}}(v)) \geq 0.
\end{eqnarray*}
\end{proof}

\begin{proposition}
Let $\mc{K}$ be a $2k$-dimensional cubical sphere, and suppose that $f_0(\lk_{\mc{K}}(v)) \in \{2k+1,2k+2\}$ for all $v \in \mc{K}$.  Then $\mc{K}$ satisfies the cubical GLBC.
\end{proposition}
\begin{proof} (\textit{sketch})
It is well known that a $(d-1)$-dimensional simplicial sphere  on $d+1$ or $d+2$ vertices can be realized as the boundary of a $d$-polytope.  Since  $g_1(\lk_{\mc{K}}(v)) \leq 1$ for all vertices $v \in \mc{K}$, the argument used to prove Theorem \ref{g-thm-small-links} gives the desired result.
\end{proof}


\section{Dehn-Sommerville Equations for Manifolds with Boundary} \label{cubicalDSsection}
 Novik and Swartz \cite{NS08} prove the following generalization of the Dehn-Sommerville equations for simplicial manifolds with boundary.

\begin{theorem} \label{NS-DehnSommerville} {\rm (\cite[Theorem 3.1]{NS08})}
Let $\Delta$ be a $(d-1)$-dimensional simplicial homology manifold with boundary.  Then
$$h_{d-i}(\Delta) - h_{i}(\Delta) = {d \choose i} (-1)^{d-i-1} \widetilde{\chi}(\Delta) - g_i(\partial \Delta),$$
for $0 \leq i \leq d$ where $g_i(\partial \Delta) = h_i(\partial \Delta)-h_{i-1}(\partial \Delta)$.
\end{theorem}

\noindent Babson, Billera, and Chan \cite[Proposition 4.1.2]{BBC} state that if $\mc{K}$ is a cubical $d$-ball, then $h_{d+1-i}\cc(\mc{K}) - h_i\cc(\mc{K}) = -g_i\cc(\partial \mc{K})$ for all $1 \leq i \leq d$.  Our goal for this section is to generalize this result and prove the following cubical analogue of Theorem \ref{NS-DehnSommerville}.

\begin{theorem}\label{cubicalDS}
Let $\mc{K}$ be a $d$-dimensional cubical manifold with boundary.  For all $1 \leq j \leq d$,
\begin{displaymath}
h_{d+1-j}\cc(\mc{K}) - h_{j}\cc(\mc{K}) = (-1)^j(-2)^{d}\widetilde{\chi}(\mc{K}) - g_j\cc(\partial \mc{K}).
\end{displaymath}
\end{theorem}

\begin{lemma}\label{newII71}
If $\mc{K}$ is a $d$-dimensional cubical complex, then
\begin{displaymath}
\sum_{j=0}^dh_j\scc(\mc{K})\lambda^{d-j} = \sum_{\stackrel{F \in \mc{K}}{F \neq \emptyset}}(-1)^{d-\dim F-1}\widetilde{\chi}(\lk_{\mc{K}}(F))(2\lambda)^{\dim F}(1-\lambda)^{d-\dim F}.
\end{displaymath}
\end{lemma}
\begin{proof}
Coarsening the Hilbert series in Theorem II.7.1 of \cite{S96} implies that if $\Delta$ is a $(d-1)$-dimensional simplicial complex, then
\begin{equation}\label{h-num-gen-function}
\sum_{j=0}^dh_j(\Delta)\lambda^{d-j} = \sum_{\sigma \in \Delta}(-1)^{d-(\dim \sigma+1)-1}\widetilde{\chi}(\lk_{\Delta}(\sigma))\lambda^{\dim \sigma+1}(1-\lambda)^{d-(\dim \sigma+1)} .
\end{equation}
The link of each vertex $v \in \mc{K}$ is a $(d-1)$-dimensional simplicial complex whose $(i-1)$-dimensional faces $\sigma \in \lk_{\mc{K}}(v)$ correspond to faces $F \in \mc{K}$ of dimension $i$ that contain $v$.  By Equation \eqref{h-num-gen-function},
\begin{equation}\label{h-num-local-gen-function}
\sum_{j=0}^dh_j(\lk_{\mc{K}}(v))\lambda^{d-j} = \sum_{\stackrel{F \in \mc{K}}{v \in F}}(-1)^{d-\dim F-1}\widetilde{\chi}(\lk_{\mc{K}}(F))\lambda^{\dim F}(1-\lambda)^{d-\dim F}.
\end{equation}
Hetyei's observation \cite[Theorem 9]{A95} that $h_j\scc(\mc{K}) = \sum_{v \in \mc{K}}h_j(\lk_{\mc{K}}(v))$, together with Equation \eqref{h-num-local-gen-function} yield the desired result.
\end{proof}

Suppose now that $\mc{K}$ is a cubical manifold with boundary.  If $F \in \partial \mc{K}$, then $\lk_{\mc{K}}(F)$ is a ball and $\widetilde{\chi}(\lk_{\mc{K}}(F)) = 0$; and if $F \in \mc{K} - \partial \mc{K}$, then $\lk_{\mc{K}}(F)$ is a simplicial sphere of dimension $d-\dim F-1$ and $\widetilde{\chi}(\lk_{\mc{K}}(F)) = (-1)^{d-\dim F-1}$.  Thus for manifolds, Lemma \ref{newII71} implies the following.
\begin{corollary} \label{manifoldII71}
Let $\mc{K}$ be a $d$-dimensional cubical manifold with boundary $\partial \mc{K}$.  Then
\begin{displaymath}
\sum_{j=0}^dh_j\scc(\mc{K})\lambda^{d-j} = \sum_{\substack{ F \in \mc{K}-\partial\mc{K} \\F \neq \emptyset}}(2\lambda)^{\dim F}(1-\lambda)^{d-\dim F}.
\end{displaymath}
\end{corollary}

Now we are ready to prove Theorem \ref{cubicalDS}.  

\textit{Proof of Theorem \ref{cubicalDS}:}

Let $\mc{K}$ be a $d$-dimensional cubical manifold with boundary $\partial \mc{K}$.  Define $\check{f}_i(\mc{K}):= f_i(\mc{K})-f_i(\partial\mc{K})$, and define $\check{h}_j\scc(\mc{K})$ by the relation $$\sum_{j=0}^d\check{h}_j\scc(\mc{K})\lambda^j = \sum_{i=0}^d\check{f}_i(\mc{K})(2\lambda)^i(1-\lambda)^{d-i}.$$  Then $\check{h}_j\scc(\mc{K}) = h_j\scc(\mc{K})-g_j\scc(\partial\mc{K})$ for all $0 \leq j \leq d$.  With this notation, Corollary \ref{manifoldII71} says
\begin{eqnarray*}
\sum_{j=0}^dh_j\scc(\mc{K})\lambda^{d-j} &=& \sum_{\substack{F \in \mc{K}-\partial\mc{K} \\ F \neq \emptyset}}(2\lambda)^{\dim F}(1-\lambda)^{d-\dim F} \\
&=& \sum_{i=0}^d\check{f}_i(\mc{K})(2\lambda)^i(1-\lambda)^{d-i} = \sum_{j=0}^d \check{h}_j\scc(\mc{K})\lambda^j.
\end{eqnarray*}
Thus $h_{d-j}\scc(\mc{K}) = \check{h}_j\scc(\mc{K}) = h_j\scc(\mc{K})-g_j\scc(\partial\mc{K})$ for all $0 \leq j \leq d$.

Observe that the defining relation $h_i\scc(\mc{K}) = h_i\cc(\mc{K}) + h_{i+1}\cc(\mc{K})$ gives
\begin{equation}\label{induction}
h_{d+1-j}\cc(\mc{K}) - h_j\cc(\mc{K}) = -\left(h_{d+1-(j-1)}\cc(\mc{K})-h_{j-1}\cc(\mc{K})\right) + h_{d-(j-1)}\scc(\mc{K}) - h_{j-1}\scc(\mc{K}).
\end{equation}
We prove the theorem by induction on $j$.

When $j=1$,
\begin{eqnarray*}
h_d\cc(\mc{K}) - h_1\cc(\mc{K}) &=& -\left(h_{d+1}\cc(\mc{K})-h_0\cc(\mc{K})\right)+h_d\scc(\mc{K})-h_0\scc(\mc{K}) \\
&=&-\left((-2)^d\widetilde{\chi}(\mc{K})-2^d\right)-f_0(\partial\mc{K}) \\
&=& (-1)(-2)^d\widetilde{\chi}(\mc{K})-g_0\cc(\partial\mc{K}).
\end{eqnarray*}
Similarly, for $j>1$, Equation \eqref{induction} and the induction hypothesis give
\begin{eqnarray*}
h_{d+1-j}\cc(\mc{K})-h_j\cc(\mc{K})
&=& -\left((-1)^{j-1}(-2)^d\widetilde{\chi}(\mc{K})-g_{j-1}\cc(\partial\mc{K})\right)-g_{j-1}\scc(\partial\mc{K}) \\
&=& (-1)^{j}(-2)^d\widetilde{\chi}(\mc{K})-g_j\cc(\partial\mc{K}).
\end{eqnarray*}
Here we use the relation $g_{j-1}\scc(\partial\mc{K}) = g_{j}\cc(\partial\mc{K})+g_{j-1}\cc(\partial\mc{K})$ from Equation \eqref{lem35} in the second line.
\hfill $\Box$

\section{Acknowledgements}
I am incredibly grateful to Isabella Novik for a number of insightful conversations during the cultivation of the ideas behind this paper.  I would also like to express sincere thanks to the anonymous referees who provided a number of helpful suggestions to improve the presentation in this paper.  This research was partially supported by a graduate fellowship from VIGRE NSF Grant DMS-0354131 at the University of Washington.

\bibliography{klee.bib}
\bibliographystyle{plain}

\end{document}